\documentclass[a4paper,11pt,reqno]{amsart}
\usepackage[T1]{fontenc}

\usepackage{amsmath}
\usepackage{tikz}
\usetikzlibrary{calc}
\usepackage{tikz}
\usetikzlibrary{decorations.markings}

\usepackage{latexsym,amsfonts,amssymb,amsthm,euscript, subfigure}
\usepackage{graphicx}
\usepackage{cleveref}
\usepackage{float}
\usepackage{color,fancybox}
\usepackage[normalem]{ulem}  
\usepackage{latexsym}
\usepackage{enumitem}
\usepackage{comment}
\usepackage[british]{babel}

\usepackage{mathrsfs}

\newcounter{hipotese}

\theoremstyle{plain}
\newtheorem{Teo}{Theorem}[section]
\newtheorem{Def}[Teo]{Definition}
\newtheorem{Ex}[Teo]{Example}
\newtheorem{Lema}[Teo]{Lemma}
\newtheorem{Prop}[Teo]{Proposition}

\newtheorem{maintheorem}{Theorem}

\theoremstyle{remark}

\newcommand{\R}{\mathbb{R}}
\newcommand{\N}{\mathbb{N}}

\newcommand{\nto}{\stackrel{n \rightarrow +\infty}{\longrightarrow}}

\begin{document}

\title{On the ergodic theory of impulsive semiflows}

\author[S. M. Afonso]{S. M. Afonso}
\address{Suzete Maria Afonso\\ Universidade Estadual Paulista (UNESP), Instituto de Geoci\^{e}ncias e Ci\^{e}ncias Exatas, C\^ampus de Rio Claro,  Avenida 24-A, 1515, Bela Vista, Rio Claro, S\~ao Paulo\\ 13506-900, Brazil}
\email{s.afonso@unesp.br} 

\author[E. Bonotto]{E. Bonotto}
\address{Everaldo Bonotto \\ Instituto de Ci\^encias Matem\'aticas e de Computa\c c\~ao\\ Universidade de S\~ao Paulo, Campus de S\~ao Carlos\\  13560-970, S\~ao  Carlos SP, Brazil}
\email{ebonotto@icmc.usp.br}

\author[J. Siqueira]{J. Siqueira}
\address{Jaqueline Siqueira\\ Departamento de Matem\'atica, Instituto de Matem\'atica, Universidade Federal do Rio de Janeiro, Caixa Postal 68530, Rio de Janeiro, RJ 21945-970, Brazil}
\email{jaqueline@im.ufrj.br}

\thanks{SMA and EB were partially supported by grant 2020/14075-6, S\~ao
Paulo Research Foundation (FAPESP).  EB was partially supported by FAPESP grant 2020/14075-6 and CNPq grants 431517/2018-5 and 310540/2019-4. JS was partially supported by  grant E-26/010/002610/2019, Rio de Janeiro Research Foundation (FAPERJ)}
\pagenumbering{arabic}

\begin{abstract} In this work, we established sufficient conditions for the existence of invariant measures for impulsive semiflows. 
In addition, we exhibit sufficient conditions to obtain a Variational Principle  and the existence and uniqueness of equilibrium states. 
        Some examples are provided to illustrate the theory developed.
  \end{abstract}

\maketitle

\tableofcontents


\section{Introduction}

The evolution in time of several phenomena studied by the sciences pre\-sents, at certain moments, abrupt changes of state, where the duration of the disturbances is null or negligible compared to the duration of the phenomenon. These abrupt changes are called impulses and they can be observed in phenomena in physics, biology, economics, control theory and information science, see \cite{Li} and \cite{Stamov}. 

Impulsive dynamical systems (IDS) can be interpreted as suitable mathe\-matical models of real world phenomena that display abrupt changes in their behavior as mentioned above. More precisely, an IDS is described by three objects: a continuous semiflow on a phase space $M$; a set $D$ contained in $M$ where the flow experiments sudden perturbations; and an impulsive function from $D$ to $M$, which determines the change in the trajectory each time it
collides with the impulsive set $D$.  In spite of their great range of applications, IDS have
started being studied from the viewpoint of ergodic theory only recently in the work  \cite{AC} and subsequently \cite{ACV}, \cite{ACS}, \cite{UFRGS}, \cite{Nelda1}. A key challenge, inherent to the dynamics, is that in general, an impulsive semiflow is not continuous, therefore one can not, a priori, even to guarantee the existence of invariant measures.  Throughout the years the ergodic theory has shown to be a useful way to obtain features of dynamical systems.  Describing the behavior of the orbits of a dynamical system can be a challenging task, especially for systems that have a complicated topological and geometrical structure, therefore, studying its behavior via invariant probability measures is a useful tool. For instance, Birkhoff's Ergodic Theorem  shows that  almost every initial condition in each ergodic component of an invariant measure has the same statistical distribution in space.

In the pioneer work \cite{AC}, the authors studied sufficient conditions to the existence of invariant measures. In the present work, we aim to provide sufficient conditions for the existence of invariant measures, that imply the ones given by Alves and Carvalho in \cite{AC}, and are somewhat easier to verify.  Our main goal in simplifying the conditions for the existence of invariant measures of impulsive semiflows is to allow and invite other works  in the ergodic theory viewpoint  allowing a future  use of the ergodic theory in its all power. In this sense, we revisit the results concerning the Variational Principle established in \cite{ACV} and the existence and uniqueness of  equilibrium states established in \cite{ACS}, where we present  sufficient conditions to verify the hypothesis regarding the impulsive set $D$ and the function $\tau^*$ (defined in~\eqref{function tau star}).

 The paper is organized as follows. In the second section we state precisely our main theorems. In Section 3, we prove Theorems \ref{main theorem 1} and \ref{main theorem 2} by using the Special Strong tube condition as the main tool. In  Section 4, we provide some examples to illustrate the developed theory. Finally, in Section 5, we establish  Theorem \ref{te.existence_uniqueness}.


\section{Main results}\label{Main results}

Let $M$ be a compact manifold endowed with the Riemannian metric $d$ and set $\R^+_0 = \{x \in \R: x \geq 0\}$. Consider a  continuous semiflow $\varphi\colon \R^+_0 \times M\to M$, a nonempty compact  set $D\subset M$ and a continuous map $I\colon D \to M$ such that $I(D)\cap D=\emptyset$. Under these conditions, we say that $(M,\varphi,D,I)$ is an \emph{impulsive dynamical system} or an IDS  for short. 

Throughout this work, we shall assume that the semiflow $\varphi$  is  the flow associated to the diffe\-rential equation
$$\dot{x} = f(x), $$
where $f\colon M\to M$ is a smooth function.

 An impulsive dynamical system generates an impulsive semiflow as follows. 
We start by defining the sequence of hitting times at the impulsive set $D$. 
The first visit of each $\varphi$-trajectory to $D$ is given by the function $\tau_1\colon M\to~[0,+\infty]$, defined as
$$
\tau_1(x)=
\left\{
\begin{array}{ccc}
\inf\left\{t> 0 \colon \varphi_t(x)\in D\right\} , &\text{if }& \varphi_t(x)\in D\text{ for some }t>0;\\
+\infty, && \text{otherwise.}
\end{array}
\right.
$$
The \emph{impulsive trajectory} $\gamma_x$  and the subsequent \emph{impulsive times} $\tau_2(x),\tau_3(x),$ $\tau_4(x),\dots$ (possibly finitely many) of a given point $x\in M$ are defined according to the following rules:
for $0\le t<\tau_{1}(x)$ we set  $\gamma_x(t)=\varphi_t(x).$

Assuming that $\gamma_x(t)$ is defined for $t<\tau_{n}(x)$ for some {$n\ge 2$}, we set
  $$\gamma_x(\tau_{n}(x))=I(\varphi_{\tau_n(x)-\tau_{n-1}(x)}(\gamma_x({\tau_{n-1}(x)}))).$$
  Defining the $(n+1)^{\text{th}}$ impulsive time of $x$ as
  $$\tau_{n+1}(x)=\tau_n(x)+\tau_1(\gamma_x(\tau_n(x))),$$
 for $\tau_n(x)<t<\tau_{n+1}(x)$, we set
         $$\gamma_x(t)=\varphi_{t-\tau_n(x)}(\gamma_x(\tau_n(x))).$$
We define the \emph{time duration} of the trajectory of $x$ as
 $\Upsilon(x)=\sup_{n\ge 1}\,\{\tau_n(x)\}.
 $
The impulsive semiflow  $\psi\colon \mathbb{R}^{+}_0\times M \to M$ associated with the impulsive dynamical system $(M,\varphi,D,I)$ is defined by $\psi(t, x) = \gamma_x(t)$, $(t, x) \in \mathbb{R}^{+}_0\times M$.
We shall denote $\psi_t(x) = \psi(t, x)$, for every  $(t, x) \in \mathbb{R}^{+}_0\times M$.

We point out that the assumption $I(D) \cap D=\emptyset$ implies that the time duration of each trajectory is infinite, see \cite[Remark 1.1]{AC}.

\begin{Def}
A probability measure $\mu$ on the Borel sets of a metric space $M$  is called  {\bf invariant} under a semiflow $\phi$ if $$\mu(\phi_t^{-1}(A)) = \mu(A)  \ \text{for every Borel set }  \ A\subset M \  \text{and for all} \ t\geq 0.  $$
\end{Def}

Recall that a continuous semiflow on a compact metric space always admit an invariant measure. However, in the lack of continuity, which is the case of impulsive semiflows, one can not guarantee that. 
Our first main result provides sufficient conditions for the existence of invariant measures under impulsive semiflows. 
\begin{maintheorem}\label{main theorem 1}
Let $\varphi\colon  M\times\mathbb{R}^{+}_0 \to M$ be a  smooth semiflow, $D$ be a smooth compact submanifold and $I\colon  D\to M$ be a continuous function. Let $\psi$ be the  impulsive semiflow generated by $(M, \varphi, D, I)$. If $D$ is a submanifold of codimension one that is transversal to the flow direction, then there exists at least one invariant measure under  the impulsive semiflow $\psi$.
\end{maintheorem}

As mentioned in the introduction, in \cite{AC} the authors provided conditions to the existence of invariant measures. Here, we replace their conditions for the above ones which are more topological and somewhat easier to check.

\subsection{Topological entropy for non-continuous flows}
 
We aim to  show that the conditions presented in Theorem \ref{main theorem 1} together with the $1$-Lipschitz assumption of $I$ and  the transversality of $I(D)$
are enough to guarantee  that a Variational Principle holds.  We start recalling a suitable definition of topological entropy in the context of impulsive semiflows. 

The topological entropy  is a number that measures how chaotic is a system: it computes the exponential growth of  the trajectories of the system. 
Moreover, two equivalent systems, namely topologically conjugated,  have the same topological entropy. Thus, the topological entropy is a useful tool in the classification of systems. 

There are different ways for defining topological entropy, some of them are equivalent in the setting of continuous flows on compact manifolds (see \cite{Viana}).  In the absence of continuity,  the classical definitions are not suitable. In order to overpass it, in \cite{ACV} the authors  defined the following concept. 

Let $\phi\colon  M\times \mathbb{R}^{+}_{0} \to M$ be a semiflow, not necessarily continuous. 

Consider a function $\tau$ assigning to each $x\in M$ a  sequence $(\tau_n(x))_{n\in A(x)}$ of nonnegative numbers, where
either $A(x)=\N$ or $A(x)=\{1,\dots,\ell\}$ for some $\ell\in \N$.
We say that $\tau$ is \emph{admissible}  if there exists $\eta>0$ such that,
for all $x\in M$  and all $n\in \N$ with $n+1\in A(x)$, we have
\begin{enumerate}
\item   $\tau_{n+1}(x)-\tau_n(x)\geq\eta$;
\item
$\displaystyle \tau_n(\psi_t(x))=
\begin{cases}\tau_n(x)-t, & \text{if $\tau_{n-1}(x)< t < \tau_n(x)$};\\
\tau_{n+1}(x),& \text{if $t=\tau_n(x)$}.
\end{cases}
$

\end{enumerate}
For each   admissible function $\tau$, $x\in M$, $t>0$ and $0<\delta<\eta/2$, we define
$$J_{t,\delta}^\tau(x) = [0,t]\setminus\bigcup_{j\in A(x)}]\tau_j(x)-\delta,\tau_j(x)+\delta[.$$

When $\tau_1(x)>t$, one has  $ J_{t,\delta}^\tau(x)= [0,t]$. 

The \emph{$ \tau$-dynamical ball} of radius $\epsilon>0$ and length $T$  centered at $x$ is defined as
$$B^\tau(x,\psi, T,\epsilon,\delta)=\left\{y\in X\colon  d(\psi_t (x),\psi_t (y))< \epsilon,\: \text{for all} \; t\in J_{T,\delta}^{\tau}(x) \right\}.$$

A set $E\subset M$ is said to be $(\psi, \tau, T, \epsilon, \delta)$-separated set  if, for each $x\in E$,  whenever $y \in B^\tau(x,\psi, T,\epsilon,\delta)$, we have $y=x$. 

Define $s^{\tau}(\psi, T, \epsilon, \delta) =\sup \left\{ \# E: E \ \mbox{is a finite} \ (\psi, \tau, T, \epsilon, \delta)\mbox{-separated set}\right\}\!,$
where $\#E$ denotes the number of elements of $E$.  
Also, define the exponential growth rate as
$$ h^{\tau}(\psi, \epsilon, \delta)= \limsup_ {T\to +\infty}\frac{1}{T}\log s^{\tau}(\psi, T, \epsilon, \delta), $$
where we assume that $\log \infty= \infty$.  Then take the limit as $\epsilon$ goes to  zero
$$ h^{\tau}(\psi,\delta)= \lim_{\epsilon \to 0^{+}} h^{\tau}(\psi, \epsilon, \delta). $$

Finally, we take the limit as $\delta$ goes to zero to define the {\bf $\tau$-topological entropy} as
$$ h^{\tau}_{top}(\psi)= \lim_{\delta \to 0^{+}} h^{\tau}(\psi, \delta) . $$

Observe that the concept is well defined since the limits above exist because the functions $\epsilon\mapsto h^{\tau}(\psi, \epsilon, \delta) $ and  $\delta \mapsto h^{\tau}(\psi, \delta) $ are decreasing. 

The $\tau$-topological entropy is a suitable concept for non-continuous  semiflow in the sense that  the definition coincides with the classical one when the semiflow is continuous \cite[Theorem A]{ACV}.

We point out that, since the work \cite{ACV}, other notions were defined in \cite{Nelda1} and \cite{Nelda2}. We also refer to  \cite{UFRGS} where the authors compare all the previous notions with the ones that they define.


Here we  establish  conditions  for which the variational principle holds for the $\tau$-topological entropy. 
Recall that we are assuming that $I(D) \cap D = \emptyset$.

\begin{maintheorem}\label{main theorem 2}
Let $\varphi\colon  M\times\mathbb{R}^{+}_0 \to M$ be a  smooth semiflow, $D$ be a smooth compact submanifold of codimension one and $I\colon  D\to M$ be an 1-Lipschitz function. If $D$ and $I(D)$ are transversal to the flow direction, then the impulsive semiflow $\psi$ generated by the impulsive dynamical system $(M, \varphi, D, I)$ satisfies  
$$ h^{\tau}_{top}(\psi)= \sup_{\mu} \{ h_{\mu}(\psi)    \}, $$
where the supremum is taken over all $\mu$ invariant probability measure under $\psi$ and $h_{\mu}(\psi)$ stands for the measure theoretical entropy of $\psi$. 
\end{maintheorem}

We refer the reader to \cite{Viana} for the definition of measure theoretical entropy.

%
%



\section{The special strong tube condition}
 This section concerns the proof of Theorem \ref{main theorem 1}. The key concept  in the proof is the special strong tube condition.  We start with some  definitions. 

For $t \geq 0$ and $x \in  M$, we
define $F(t, x)= \{y \in M: \varphi(t, y) = x \}$
and, for $\Delta \subset [0,  \infty)$ and $Y \subset  M$, we define
$$
F( \Delta, Y) = \bigcup_{t \in \Delta, x \in Y} F( t, x).
$$
A point $x \in  M$ is
called an \emph{initial point}, if $F(t, x) = \emptyset$
for all $t>0$.

\begin{Def}\label{Tube} \rm
 A closed subset $S \subset M$ containing $x\in M$ is
called a {\it section} through $x$ if there exist $\lambda>0$ and a closed subset $L$ of $M$
such that:
\begin{itemize}
\item[{\bf (a)}] $F(\lambda, L)=S$;
\item[{\bf (b)}] $F([0,2\lambda], L)$ contains a neighborhood of $x$;
\item[{\bf (c)}] $F(\nu, L)\cap F(\zeta, L)=\emptyset$, if $0\leqslant \nu<\zeta\leqslant 2\lambda$.
\end{itemize}
The set $F([0,2\lambda], L)$ is called a $\lambda$-{\it tube} 
and the set $L$ is called a {\it bar}.
\end{Def}

\begin{Def}\label{SSC}
Let $(M,\varphi,D,I)$ be an IDS. We say that a point $x\in D$ satisfies the {\it strong tube
condition $($STC$)$}, if there exists a section $S$ through $x$ such that
$S= F([0,2\lambda], L)\cap D$. Furthermore, a point $x \in D$ satisfies the {\it special strong tube
condition $($SSTC$)$} if it satisfies STC and the  $\lambda$-tube $F([0,2\lambda], L)$ is such that $F( [0, \lambda], L) \cap I(D) = \emptyset$. 

If every point in $D$ satisfy  the SSTC, then we say that $D$ satisfies SSTC.
\end{Def}

\begin{Def}
Let $S$ be a codimension one submanifold of  a manifold $M$. We say that $S$ is transversal at $p\in S$ if $f(p)\notin T_pM$.  If $f(p) \notin T_pM $ for all $p\in S$, then $S$ is said to be transversal to the flow direction.  
\end{Def}

In the next proposition we prove that, under general topological conditions, the impulsive set and its image satisfy the  SSTC.

\begin{Prop}\label{strong tube}
If $D$ is a smooth compact submanifold of $M$  of codimension $1$ which is transversal to the direction of the  flow,  then  $D$ satisfies the special strong tube condition. 
\end{Prop}

Proposition \ref{strong tube} is an  extension of \cite[Theorem 9]{EMTR}. Here, we consider general manifolds in opposite to the subsets of Euclidean spaces considered in \cite{EMTR}.  We will  follow closely the ideas of the proof presented in \cite{EMTR}.
First we will need  to state a couple of lemmas from  Classical Differential Topology.   The reader may consult the proofs in \cite{Lee}.

\begin{Lema}\cite[Proposition 5.16]{Lee}\label{Lee result 1} 
Let $S$ be a subset of a smooth m-manifold $M$. Then $S$ is an embedded $k$-submanifold of $M$ if and only if every point of $S$ has a neighborhood $U$ in $M$ such that $U\cap S$ is a level set of a smooth submersion $\Phi\colon  U \mapsto \R^{m-k}$.
\end{Lema}

\begin{Lema}\cite[Proposition 5.38]{Lee} \label{Lee result 2} 
Suppose $M$is a manifold and $S\subset M$ is an embedded submanifold. If $\Phi\colon  U \mapsto N$ is any local defining map for $S$, then $T_pS= {\rm Ker} d\Phi_p\colon  T_pM \mapsto T_{\Phi(p)}N $ for each $p\in S\cap U$.

\end{Lema}

Now we can proceed to the proof of Proposition \ref{strong tube}. 

\begin{proof}[Proof of Proposition~\ref{strong tube}]

Let $x \in D$.  By Lemma \ref{Lee result 1},  there exists a neighborhood $U$ in $M$, with $x \in U\cap D$, such that $U\cap D$ is a level set of a smooth submersion $\Phi\colon U \mapsto \R$, that is, there exists $c \in \R$ such that $U\cap D = \Phi^{-1}(c)$.

\vspace{.2cm}

\textbf{Claim 1:} There exists $\epsilon_x > 0$ such that
$\varphi((0, \epsilon_x),x) \cap D= \emptyset$ and $F((0, \epsilon_x), x)\cap D = \emptyset$.

Indeed, suppose by contradiction that there exists a sequence of positive real numbers $\{t_n\}_{n\in\mathbb{N}}$ with $t_n \nto 0$ such that $\varphi(t_n, x) \in D$, $n\in\N$. Let $\delta > 0$ be such that $\varphi([0, \delta],x) \subset U$. We may assume that 
 $t_n \in [0, \delta]$ for every $n \in \N$. Define $h\colon [0, \delta] \to \mathbb{R}$ by $h(t) = \Phi(\varphi(t,x)) - c$. Note that $h(0) = h(t_n) = 0$ for all $n\in\N$ and, by the Rolle Theorem, there exists $\tau_n \in (0,t_n)$ such that
$$h'(\tau_n)=\nabla \Phi(\varphi(\tau_n,x))f(\varphi(\tau_n,x))=0, \; n\in\N.$$
As $n\to +\infty$, we obtain $\nabla \Phi(x)f(x)=0$. By Lemma \ref{Lee result 2}, we conclude that $f(x) \in T_xD$ which is a contradiction since $D$ is transversal to the flow direction. 

Analogously, if there exists a sequence $\{t_n\}_{n\in\mathbb{N}}\subset \R_+$,
 $t_n \nto 0$, such that $F(t_n,x)\cap D \neq \emptyset$ for all $n\in\N$, then we can also assume that there exists $\delta > 0$ such that $F([0, \delta],x) \subset U$ and $t_n<\delta$ for all $n\in\N$. For each $n\in \N$, let $y_n\in F(t_n,x)\cap D$ and define the mapping $h_n\colon [0,t_n] \to \mathbb{R}$ by $h_n(t) = \Phi(\varphi(t, y_n)) - c$. Using again the Rolle Theorem we obtain $\tau_n \in (0, t_n)$ with $h_n'(\tau_n) = 0$, $n\in\N$,
 that yields $\nabla \Phi(x)f(x)=0$ which is a contradiction.

 \vspace{.2cm}
 
 \textbf{Claim 2:} If $t>0$ is such that $\varphi((0,t],x)\subset U\backslash D$ and $\varphi( [-t,0), x) \subset U\backslash D$ then $(\Phi(\varphi( t, x))-c)(\Phi(\varphi( -t, x))-c)<0$.

Indeed, note that $(\Phi(\varphi( t, x))-c)(\Phi(\varphi(s,x))-c)>0$ for all $0<s\leqslant t$, because if there exists $0<s < t$ such that $$(\Phi(\varphi( t, x))-c)(\Phi(\varphi( s, x))-c)<0,$$
then one can obtain $\tau\in (s, t)$ such that $\Phi(\varphi(\tau, x))=c$, that is, $\varphi( \tau, x)\in D$ which is a contradiction. Analogously, we have $$(\Phi(\varphi(-t, x))-c)(\Phi(\varphi(-s, x))-c)>0 \quad \mbox{for all}\;\; 0 < s \leqslant t.$$

Now, suppose the opposite, i.e., that $(\Phi(\varphi(x,t))-c)(\Phi(\varphi(x, -t))-c)>0$. Then, either $\Phi(\varphi(x, [-t,0)\cup(0,t]))> c$ or $\Phi(x, \varphi([-t,0)\cup(0,t]))<c$ which implies that $\nabla\Phi(x)f(x)=0$ and gives also a contradiction.

 \vspace{.2cm}
 
 \textbf{Claim 3:} $D$ satisfies SSTC.

We may assume that $U\cap D= \Phi^{-1}(c)$ is connected. Using the same argument as in \textbf{Claim 1}, we may obtain a compact set $S \subset U\cap D$, $x \in S$, and $\lambda>0$ be such that $\varphi((0, 2\lambda], S) \subset U\backslash D$ and $F((0, 2\lambda], S) \subset U\backslash D$.  Since $D$ is compact and $D \cap I(D) = \emptyset$, we may choose $\lambda > 0$ such that $\varphi((0, 2\lambda], S) \cap I(D) = \emptyset$. Define $L=\varphi( \lambda, S)$. Note that $F([0,\lambda],L) \cap I(D) = \emptyset$.
 We are going to show that $F([0,2\lambda], L)$ is a $\lambda$-tube through $x$.

It is clear that $F( \lambda, L) = S$. Now, suppose by contradiction that the tube $F([0,2\lambda],L)$ does not contain a neighborhood of $x$, i.e., there exists a sequence $x_n\nto x$ such that $x_n\notin F([0,2\lambda],L)$ for all $n\in\N$. 

\textbf{Claim 2} implies that there exists $0<s<\lambda$ such that
\[(\Phi(\varphi( s, x))-c)(\Phi(\varphi( -s, x))-c)<0.\] Therefore, there is $n_0\in\N$ so that \[\varphi( [-s,s], x_n)\subset U\quad \textrm{and}\quad (\Phi(\varphi( s, x_n))-c)(\Phi(\varphi( -s, x_n))-c)<0\] for all $n\geq n_0$. 

Thus, for $n\geq n_0$, there is $\tau_n\in(-s,s)$ such that $\Phi(\varphi( \tau_n, x_n))=c$, that is, $\varphi(\tau_n, x_n)\in D$, which gives a contradiction.

To finish the proof we show that $F( \alpha, L)\cap F( \beta, L)=\emptyset$ for  $0\leqslant \alpha < \beta \leqslant 2\lambda$. Suppose, again by contradiction, that there exists $y\in F( \alpha, L)\cap F(\beta, L)$ for some $0\leqslant \alpha < \beta \leqslant 2\lambda$. Then $a = \varphi(y, \alpha-\lambda)\in D$ and $b = \varphi(y, \beta-\lambda)\in D$. Consequently, $\varphi(a, \beta-\alpha)=b \in D$ and this is a contradiction since $a\in D$ and $\varphi(S, (0, 2\lambda]) \cap D = \emptyset$. This completes the proof.
\end{proof}

The special strong tube condition provides an useful property to the impulsive semiflow. 

\begin{Prop}\cite[Proposition 2.6]{EMAR}\label{Prof aux 1} Let $(M,\varphi,D,I)$ be an IDS such that $I(D) \cap D=\emptyset$ and assume that $ D$ satisfies SSTC with $\lambda$-tube $F([0,2\lambda], L)$. Then $\psi_t\left(M) \cap  F([0,\lambda], L\right) = \emptyset$ for $t > \lambda$.
\end{Prop}

We recall that given a semiflow $\phi\colon  M\times \R_0^+ \to M$, not necessarily continuous, a point   $x \in M$ is said to be \textbf{non-wandering}  if, for every neighborhood $U$ of $x$ and any $T>0$, there exists $t \geq T$ such that 
\[
\phi^{-1}_t(U) \cap U \neq \emptyset.
\]

Let $\Omega_{\phi} = \{x \in X\colon \, x \; \text{is non-wandering for} \; \phi\}$ be the non-wandering set. The compactness of $M$ implies that the non-wandering set is also compact. In the case that the semiflow is continuous, the non-wandering set is also invariant under the semiflow, that is, 
$$\phi_t({\Omega_{\phi}}) \subset \Omega_{\phi} \quad  \text{for all} \ t\geq 0.$$

The invariance of the non-wandering set is not true in general. However, the following characterization for an impulsive semiflow allows us to obtain an invariance outside of the impulsive set. 

First, for $x\in M$, the set
\[
J(x)= \{y\in M\colon \, \text{there are sequences} \; \; \; \{x_n\}_{n\in \N} \subset M \; \; \; \text{and} \; \; \; \{t_n\}_{n\in \N}\subset\R_0^+
\]
\[
\text{such that} \; \; \; x_n\stackrel{n \rightarrow +\infty}{\longrightarrow} x, \; \; \; t_n \stackrel{n \rightarrow +\infty}{\longrightarrow} +\infty \; \; \; \text{and} \; \; \; \psi(t_n, x_n)\stackrel{n \rightarrow +\infty}{\longrightarrow} y\}
\]
denotes the
prolongation of the positive limit set of $x$ with respect to $\psi$.

\begin{Lema}\cite[Theorem 3.5]{EF}\label{nao errante} Let $(M, \varphi, D, I)$ be an IDS and let $\psi$ be the impulsive semiflow.  A point $x \in \Omega_{\psi}$ if and only if $x \in J(x)$.
\end{Lema}

\begin{Lema}\label{property of I} Let $(M, \varphi, D, I)$ be an impulsive system  
such that each point  on $D$ satisfies the SSTC. If $x \in \Omega_{\psi} \cap D$ then $I(x) \in \Omega_{\psi}$.
\end{Lema}
\begin{proof}
 Let $x \in \Omega_{\psi} \cap D$. By  Lemma~\ref{nao errante} there are sequences $\{z_n\}_{n \in \mathbb{N}} \subset M$ and $\{t_n\}_{n \in \mathbb{N}}\subset \mathbb{R}_0^+$ such that $z_n \stackrel{n \rightarrow +\infty}{\longrightarrow} x$, $t_n \stackrel{n \rightarrow +\infty}{\longrightarrow} \infty$ and 
\[
w_n = \psi(t_n,z_n) \stackrel{n \rightarrow +\infty}{\longrightarrow} x.
\]
Since $D$ satisfies SSTC, it follows from Proposition \ref{Prof aux 1} that 
\[
\psi(t_n, z_n) \in F( (\lambda, 2\lambda], L) \quad  \mbox{for $n$ sufficiently large.}
\]
 Consequently, $\tau_1(\psi(t_n,z_n)) \stackrel{n \rightarrow +\infty}{\longrightarrow} 0$ and
 \[
\psi( t_n + \tau_1(\psi(t_n,z_n)), z_n) \stackrel{n \rightarrow +\infty}{\longrightarrow} I(x).
 \]
 Hence, $I(x) \in \Omega_{\psi}$.
\end{proof}

Since we are assuming that $I(D)\cap D=\emptyset$, it follows by Lemma \ref{property of I} that $I(\Omega_{\psi}\cap D) \subset \Omega_{\psi}\backslash D$.  Thus \cite[Theorem B]{AC} can be rewritten as follows. 

\begin{Prop}\label{inv n.errante}
Given the impulsive semiflow $\psi$  generated by the IDS $(M,\varphi, D, I)$ we have 
$\psi_t(\Omega_{\psi} \backslash D ) \subset \Omega_{\psi} \backslash D$ for all $t \geq 0$.
\end{Prop}

Proposition \ref{inv n.errante} establishes that outside from the impulsive set $D$ the non-wandering set is invariant under the semiflow.

Given  $\xi>0$  we define
\begin{equation*}\label{eq.dxi}
D_\xi=\bigcup_{x\in D }\{\varphi_t(x) \colon  0<t<\xi\} \quad \mbox{and}  \quad  X_\xi=M\backslash (D_\xi\cup D).
\end{equation*}

Since $D$ is compact, $I$ is continuous and $I(D)\cap D=\emptyset$, it follows for a sufficiently small $\xi$ that $I(D)\cap D_\xi=\emptyset$. Therefore,   the set  $X_\xi$ is  invariant under $\psi$, that is,
\begin{equation*}\label{eq.forinva}
\psi_t(X_\xi) \subseteq X_\xi,\quad \mbox{for all} \ t \geq 0.
\end{equation*}
We consider the function
$$\tau^*_\xi\colon X_\xi \cup D \to [0,+\infty]$$
defined as
\begin{equation}\label{function tau star}
\tau^*_\xi(x)=
\begin{cases}
\tau_1(x), &\text{if $x\in X_\xi$};\\
0, &\text{if $x\in D$}.
\end{cases}
\end{equation}

\begin{Prop}\label{cont tau} Assume that $ D$ satisfies STC. The function $\tau^*_\xi$ is continuous on $X_\xi \cup D$. 
\end{Prop}
\begin{proof} According to \cite[Theorem 3.8]{Ciesielski}, $\tau^*_\xi$ is continuous on $X_\xi$. Now, let $x \in D$ and $\{x_n\}_{n \in \N}\subset X_\xi \cup D$ be a sequence such that $x_n \stackrel{n \rightarrow +\infty}{\longrightarrow}  x$.

 Since $D$ satisfies STC and $x \in D$, there exists a tube $F( [0, 2\lambda],L)$ through $x$ given by a section $S \subset D$ such that
$S=F(L,[0,2\lambda])\cap D$. We may assume that $x_n \in F([\lambda, 2\lambda],L)$ for every $n \in \N$. Let $\lambda_n \in [\lambda, 2\lambda]$ be such that $\varphi( \lambda_n, x_n) \in L$, $n \in \N$. Since $F( \lambda, L) = S$, we have
\begin{equation}\label{eq conver}
\varphi\left(\lambda_n - \lambda, x_n \right) \in S \subset D, \; n \in \N.
\end{equation}
Let $\{\lambda_{n_k}\}_{k \in \mathbb{N}}$ be a subsequence of $\{\lambda_{n}\}_{n \in \mathbb{N}}$.  We may assume, up to a subsequence, that $\lambda_{n_k} \stackrel{k \rightarrow +\infty}{\longrightarrow}  t_0 \in [\lambda, 2\lambda]$. If $t_0 \neq \lambda$, using \eqref{eq conver}, 
 we obtain 
\[
\varphi\left( t_0 - \lambda , x\right) \in D
\]
and $\tau_{\xi}^*(x) = t_0 - \lambda \leq \lambda$, which is a contradiction. Hence,  $\lambda_{n_k} \stackrel{k \rightarrow +\infty}{\longrightarrow}  \lambda$. In conclusion, the sequence $\{\lambda_{n}\}_{n \in \mathbb{N}}$ converges to $\lambda$. Since $\varphi\left( \lambda_n - \lambda, x_n \right) \in  D$ for all $n \in \N$, we get
\[
\tau_{\xi}^*(x_n) = \lambda_n - \lambda \stackrel{n \rightarrow +\infty}{\longrightarrow}  0 = \tau_{\xi}^*(x).
\]
 Hence, $\tau^*_\xi$ is continuous on $X_\xi \cup D$.
\end{proof}

Now, the proof of Theorem \ref{main theorem 1} follows immediately by  \cite[Theorem A]{AC},  where the continuity of $\tau^*_\xi$ and the invariance of the non-wandering set, up to the impulsive set,  are required. Here, we obtain the regularity of the IDS implies the aforementioned conditions, according to Proposition~\ref{cont tau} and Proposition~\ref{inv n.errante}.

In order to prove Theorem \ref{main theorem 2}, we will need to the following auxiliary result.

\begin{Teo}\cite[Theorem B]{ACV} \label{Teo Alves, Carvalho, V\'asquez}   
 Let $\psi$ be the semiflow of an impulsive dynamical system $(M, \varphi, D,I)$ such that  $ I(D) \cap D = \emptyset$,  $I$ is an 1-Lipschitz function, $ D$ satisfies a half-tube condition, I(D) is transverse to the direction flow and $\tau^*_{\xi}$ is continuous. Then 
$$ h^{\tau}_{top}(\psi)= \sup_{\mu} \{ h_{\mu}(\psi)    \} .$$
\end{Teo}

As developed previously,  we proved that the hypothesis of  Theorem \ref{main theorem 2} imply the continuity of $\tau^*_{\xi}$ (Proposition~\ref{cont tau}) and thus, by Theorem \ref{Teo Alves, Carvalho, V\'asquez}, Theorem \ref{main theorem 2} holds.



\section{Examples}\label{exemplos}

In this section, we present some examples of systems for which we can apply  our results.

\begin{Ex} Consider $M=  \{ (r\cos \theta, r\sin\theta):\, r\in [0,\infty), \theta \in [0,2\pi]\}$ and the smooth vector field on $M$
\begin{equation}\label{ode1}
\left\{
\begin{array}{l}
 \theta'=0,   \\
  r' = -r.
\end{array}
\right.
\end{equation}
The semiflow $\varphi\colon  \mathbb{R}_0^{+}\times M \to M$ of this vector field has trajectories that converge towards the origin. 

 Defining the impulsive set by $D = \{(r, \theta) \in M: \,  r = 1, \, \theta\in [0, 2\pi]\}$ and the impulsive function on $D$ by
 $I(1, \theta) = (2, \theta)$, $\theta \in [0, 2\pi]$,  we have $D$ and  $I(D)$ are both transversal to the direction of the flow $\varphi$, and, hence, the impulsive semiflow $\psi$ gene\-rated by $(M, \varphi, D, I)$  admits at least one invariant probability  measure and, moreover, a variational principle holds. 
\end{Ex}

\begin{Ex}  Consider a solar heating system consisting of a solar collector and a solar storage,
governed by the system
\begin{equation}\label{Eq	SH 1}
\left\{
\begin{array}{ccc}
 \dfrac{dT_c}{dt} = -a_1T_c + a_2T_s + a_3 \vspace{.2cm}\\ 
 \dfrac{dT_s}{dt} =  a_4T_c - a_5T_s + a_6,
\end{array}
\right.
\end{equation}
where $T_c$ denotes the homogeneous temperature and outlet (fluid) temperature of the collector $(^oC)$ and $T_s$ represents the (fluid) temperature of the mixed storage $(^oC)$, subject to the impulsive condition $I = (I_1, I_2)\colon D \to I(D)$ with $D = \{(x, y) \in \R^2: \, (x- T_c^*)^2 + (y- T_s^*)^2  = \beta^2\}$ and
\[
(I_1(x,y)- T_c^*)^2 + (I_2(x,y)- T_s^*)^2  = (\beta+\epsilon)^2, \ \ (x, y) \in D,
\]
where $T_c^* = \dfrac{a_5a_3 + a_2a_6}{a_1a_5 - a_2a_4}$ and $T_s^* = \dfrac{a_4a_3 + a_1a_6}{a_1a_5 - a_2a_4}$. The constants $a_i > 0$, $i=1,\ldots, 6$, satisfy $a_1a_5 - a_2a_4 > 0$, and $\beta, \epsilon > 0$. The model \eqref{Eq	SH 1} without impulses is investigated in \cite{SKHG}. Moreover, as presented in \cite{SKHG}, the constants $a_i > 0$, $i=1,\ldots, 6$, depend on the collector/storage surface areas, specific heat capacity of the collector/storage fluid,  global solar irradiance on the collector surface,  heat loss coefficient of the storage,  environment temperature of the collector, cold inlet (fluid) temperature of the collector/storage, (fluid) temperature of the mixed storage,  environment temperature of the storage, uncertainty of the solar irradiance measurement,  overall heat loss coefficient of the collector, uncertainty of the measurement of the temperatures, calculated uncertainty of the collector temperature,  calculated uncertainty of the storage temperature, uncertainty of the flow rate measurement, volumetric pump flow rate in the collector, collector volume,  volumetric flow rate of the consumption load,  storage volume, optical efficiency of the collector, and  collector/storage fluid density.

The system \eqref{Eq SH 1} defines a continuous dynamical system $(\R^2, \phi)$ such that
\[
\lim_{t \to \infty} \phi_t(x_0, y_0) =
\left(
\begin{array}{ccc}
T_c^* \vspace{.2cm} \\
 T_s^*
\end{array}
\right),
\]
for every $(x_0, y_0) \in \R^2$, i.e., the solutions of \eqref{Eq	SH 1} converges asymptotically to its equilibrium point, see \cite{SKHG} for more details.  

When we consider the impulsive system, the temperature pair $(T_c, T_s)$ can be controlled over time 
 in such a way that it stays in a range close or not to the equilibrium. This control depends on the choices of variables 
 $\beta$, $\epsilon$ and $I$. 

Note that  $D = g^{-1}(0)$, where $0$ is a regular value of $g(x,y) =  (x- T_c^*)^2 + (y- T_s^*)^2  - \beta^2$. Thus, $D$
 is a compact smooth submanifold of codimension one with $I(D) \cap D = \emptyset$. Furthermore,  since the solutions of the non-impulsive system \eqref{Eq	SH 1} converges asymptotically to $(T_c^*, T_s^*)$, we have $D$  is transversal to the flow direction. Consequently, by Theorem \ref{main theorem 1}, there exists at least one invariant measure under  the impulsive semiflow $\psi$ generated from \eqref{Eq	SH 1}, $D$ and $I$. Let $\mu$ be an invariant measure  under $\psi$. Then, Poincar\'e Recurrence Theorem guarantees  that $\mu$-almost every initial condition is recurrent.   
\end{Ex}

\section{Equilibrium states}  In this section, we revisit the result on the existence and uniqueness of equilibrium states for potential functions established in \cite{ACS}.


As defined previously, let $\psi$ be the semiflow  of an impulsive dynamical system $(M,\varphi,D,I)$. For a given $r>0$, we denote the set $B_r(D)$ by the $r$-neighborhood of $D$ in $M$.

\begin{Def} The semiflow  $\psi$ is called  {\bf positively expansive} on $M$, if for every $\delta>0$ there exists $\epsilon>0$ such that, if  $x, y \in M$ and a continuous map $s: \mathbb{R}_0^+ \rightarrow \mathbb{R}_0^+$ with $s(0)=0$ satisfy
$d\,(\psi_t(x),\psi_{s(t)}(y)) < \epsilon$  
for all $t\ge 0$
such that $\psi_t(x),\psi_{s(t)}(y)\notin B_\epsilon(D)$, then
$y = \psi_t(x)$
for some $0<t < \delta$.
\end{Def}
\begin{Def} The semiflow $\psi$ has the {\bf specification property} on $M$, if for all $\epsilon>0$ there exists $L>0$ such that,
for any sequence $x_0,\dots, x_n$ of points in $M$ and any sequence $0\le t_0<\cdots< t_{n+1}$ such that
$t_{i+1}-t_i \geq L$ for all $0\le i\le n$, there are  $y \in M$ and  $r:\mathbb{R}^+_0 \rightarrow \mathbb{R}^+_0$, which is  constant on each interval $[t_i, t_{i+1}[$, whose values depend only on $x_0,\dots, x_n$, that also satisfy
$$ r([t_0, t_1[)=0\quad \text{and}\quad
|r([t_{i+1}, t_{i+2}[) - r([t_i, t_{i+1}[)| < \epsilon,
$$
for which
$$d\,(\psi_{t+r(t)}(y),\psi_{t-t_i}(x_i)) < \epsilon, \quad  \text{for all} \; \; t \in [t_i, t_{i+1}-L[ \;\; \text{and} \;\; 0\le i \le n.$$
In addition, the specification is called {\bf periodic} if we can always choose $y$ periodic with the minimum period in $[t_{n+1}-t_0 - n\epsilon, t_{n+1}-t_0 + n\epsilon]$.
\end{Def}


\begin{Def}\label{est eq}
A probability measure $\mu$, $\psi$-invariant, is an {\bf equilibrium state} for $\psi$ with respect to  a  potential function $f: M \to \mathbb{R}$ when it satisfies:
          $$h_{\mu} (\psi) + \int f \, d\mu  = \sup_\eta \left\{ h_{\eta} (\psi) + \int f \, d\eta \right\}  $$ 
where the supremum is taken over all $\psi$-invariant probability measures $\eta$. 
\end{Def}


Let $V^*(\psi)$ be the set of all continuous maps $f\colon M \rightarrow \mathbb{R}$ satisfying:
\begin{enumerate}
\item $f(x)=f(I(x))$ for all $x\in D$;
\item there are $K>0$ and $\epsilon > 0$ such that for every $t>0$  we have
\begin{equation}\label{temK*}
\left|\int_0^t f(\psi_s(x))\,ds - \int_0^t f(\psi_s(y))\,ds\right|< K,
\end{equation}
whenever
$d\,(\psi_{s}(x), \psi_s(y)) < \epsilon$ for all $ s \in [0,t]$ such that $\psi_s(x),\psi_s(y)\notin B_\epsilon(D)$.
\end{enumerate}

Applying  Theorem A in \cite{ACS}, we can state the result on the existence and uniqueness of equilibrium states for potential functions in $V^*(\psi)$ without assuming that $I(\Omega_{\psi}\cap D) \subset \Omega_{\psi}\backslash D$ and that $\tau_{\xi}^*$ is continuous on $X_\xi \cup D$.



\begin{maintheorem}\label{te.existence_uniqueness}
 Let $\varphi\colon  M\times\mathbb{R}^{+}_0 \to M$ be a  smooth semiflow, $D$ be a smooth compact submanifold of codimension one and $I\colon  D\to M$ be an 1-Lipschitz function. If $D$ and $I(D)$ are transversal to the flow direction, $\psi$  is positively expansive and has the periodic specification property in  $\Omega_\psi \backslash D$, then any potential function $f \in V^*(\psi)$ has an equilibrium state. Moreover, if $\dim(M)<\infty$  and there is $k>0$ such that $\#I^{-1}(\{y\})\le k$ for every $y \in I(D)$, then the equilibrium state is unique.
 \end{maintheorem}
 

 \section*{Acknowledgements}

The authors thank the anonymous referee for the helpful comments and suggestions.

\newpage

\end{document}